\documentclass{amsart}

\usepackage{amssymb,amsmath,amsfonts,mathptmx,cite,mathrsfs,url}

\usepackage{gastex}

\newtheorem{theorem}{Theorem}[section]
\newtheorem{problem}[theorem]{Problem}
\newtheorem{proposition}[theorem]{Proposition}
\newtheorem{lemma}[theorem]{Lemma}
\newtheorem{corollary}[theorem]{Corollary}

\theoremstyle{remark}
\newtheorem{remark}{Remark}

\usepackage{gastex}

\newcommand{\ais}{ai-semi\-ring}
\def\softd{{\leavevmode\setbox1=\hbox{d}%
    \hbox to 1.05\wd1{d\kern-0.4ex{\char039}\hss}}}

\numberwithin{equation}{section}

\begin{document}

\title{Semiring identities of finite inverse semigroups}
\thanks{Supported by the Russian Science Foundation (grant No. 22-21-00650)}
\author[S. V. Gusev]{Sergey V. Gusev}
\address{Institute of Natural Sciences and Mathematics\\
Ural Federal University\\ 620000 Ekaterinburg, Russia}
\email{sergey.gusb@gmail.com}

\author[M. V. Volkov]{Mikhail V. Volkov}
\address{Institute of Natural Sciences and Mathematics\\
Ural Federal University\\ 620000 Ekaterinburg, Russia}
\urladdr{http://csseminar.kmath.ru/volkov/}
\email{m.v.volkov@urfu.ru}

\date{}

\begin{abstract}
We study the Finite Basis Problem for finite additively idempotent semirings whose multiplicative reducts are inverse semigroups. In particular, we show that each additively idempotent semiring whose multiplicative reduct is a nontrivial rook monoid admits no finite identity basis, and so do almost all additively idempotent semirings whose multiplicative reducts are combinatorial inverse semigroups.
\end{abstract}



\maketitle

\section{Introduction}

\subsection{Background and motivation}
An \emph{additively idempotent semiring} (\ais, for short) is an algebra $\mathcal{S}=(S, +, \cdot)$ of type $(2, 2)$ such that the additive reduct $(S,+)$ is a \emph{semilattice} (that is, a commutative idempotent semigroup), the multiplicative reduct $(S,\cdot)$ is a semigroup, and multiplication distributes over addition on the left and on the right, that is, $\mathcal{S}$ satisfies the identities $x(y+z)\approx xy+xz$ and $(y+z)x\approx yx+zx$. The class of \ais{}s is extensively studied in the literature as it includes many objects of importance for computer science, idempotent analysis, tropical geometry, and algebra such as, e.g., semirings of binary relations \cite{Jip17}, syntactic semirings of languages \cite{Polak01}, tropical semirings \cite{Pin98}, endomorphism semirings of semilattices \cite{JKM09}.

Recall that a set $\Sigma$ of identities valid in an algebra $\mathcal A$ is said to be an \emph{identity basis} for $\mathcal A$ if $\Sigma$ infers all identities holding in $\mathcal A$. An algebra $\mathcal A$ is \emph{finitely based} (FB) if it admits a finite identity basis; otherwise $\mathcal A$ is called \emph{nonfinitely based} (NFB). The question of classifying algebras of a certain sort with respect to the property of being FB/NFB is known as the Finite Basis Problem (FBP). Being very natural by itself, the FBP has also revealed several interesting and unexpected relations to many issues of theoretical and practical importance. In the study of various \ais{}s, the FBP has attracted considerable attention lately. In particular, we mention \cite{AM11,AEI03} and a series of Dolinka's papers~\cite{Dolinka07,Dolinka09a,Dolinka09b,Dolinka09c}. A recent breakthrough in the area is the paper \cite{JackRenZhao} by Jackson, Ren, and Zhao, who provided a wealth of surprising examples of finite NFB \ais{}s $(S, +, \cdot)$, including those whose multiplicative reducts $(S,\cdot)$ are FB semigroups.

Despite the great progress in~\cite{JackRenZhao}, the ultimate goal of classifying FB and NFB finite \ais{}s has not been achieved yet. The final section of~\cite{JackRenZhao} contains an extensive list of problems which the authors of that paper feel provide useful directions towards this goal. In the present note, we address one of these problems and exhibit new families of finite NFB \ais{}s. We specify the problem and describe our results in more detail in Sect.~\ref{sec:formulation}, after giving necessary definitions.

\subsection{Overview of main results}
\label{sec:formulation}

Recall that elements $x,y$ of a semigroup $(S,\cdot)$ are said to be \emph{inverses} of each other if $xyx=x$ and $yxy=y$. A~semigroup $(S,\cdot)$ is called \emph{inverse} if every its element has a unique inverse; the inverse of an element $x\in S$ is denoted by $x^{-1}$. Inverse semigroups can therefore be thought of as algebras of type (2,1) where the unary operation is defined by $x\mapsto x^{-1}$.

In every inverse semigroup $(S,\cdot,{}^{-1})$, the relation
\[
\le_{\mathrm{nat}}:= \{(x,y)\in S\times S\mid x=xx^{-1}y\}
\]
is a partial order compatible with both multiplication and inversion; see \cite[Section II.1]{Pet84} or \cite[pp. 21--23]{Law99}. This order is referred to as the \emph{natural partial order}. Given a subset $H\subseteq S$, the infimum of $H$ with respect to $\le_{\mathrm{nat}}$ may not exist, but if $\inf H$ exists, then so do $\inf(sH)$ and $\inf(Hs)$ for every $s\in S$, and one has $\inf(sH) = s(\inf H)$ and $\inf(Hs) = (\inf H)s$ \cite[Proposition 1.22]{Sch73}; see also \cite[Proposition 19]{Law99}. Therefore, if an inverse semigroup $(S,\cdot,{}^{-1})$ is such that the partially ordered set $(S,\le_{\mathrm{nat}})$ is an inf-semilattice, then letting
\begin{equation}
\label{eq:inf}
x\,+_{\mathrm{nat}}\,y:=\inf\{x,y\}
\end{equation}
for all $x,y\in S$ makes $(S,+_{\mathrm{nat}},\cdot)$ be an \ais. Such \ais{}s are called \emph{naturally semilattice-ordered inverse semigroups} in \cite{JackStokes}. The following is Problem~7.7(3) from \cite{JackRenZhao}:

\begin{problem}
\label{prob:jackson}
Which finite naturally semilattice-ordered inverse semigroups are finitely based, in either of the signatures $\{+,\cdot\}$ or $\{+,\cdot,0\}$?
\end{problem}
(The appearance of the alternative signature $\{+,\cdot,0\}$ is justified by the observation that any finite naturally semilattice-ordered inverse semigroup has zero: if 0 is the least element under the natural partial order, then it is easy to see that 0 is also the multiplicative zero.)

It is Problem~\ref{prob:jackson} that has given rise to the present paper. To describe our contribution, we have to recall a few further notions. A semigroup is called \emph{combinatorial} if all of its subgroups are trivial and \emph{periodic} if all of its monogenic subsemigroups are finite. Leech \cite[Example 1.21(d), item (iv)]{Leech95} observed that if an inverse monoid $(S,\cdot,{}^{-1},1)$ is periodic and combinatorial, then $(S,\le_{\mathrm{nat}})$ is an inf-semilattice. Of course, the requirement of being a monoid is not essential: if an inverse semigroup $(S,\cdot,{}^{-1})$ is periodic and combinatorial then so is the inverse monoid $(S^1,\cdot,{}^{-1},1)$ obtained by adjoining a fresh element 1 to the carrier set $S$ and letting $1\cdot x=x\cdot 1:=x$ for all $x\in S^1$ and $1^{-1}:=1$. Thus, every periodic (in particular, finite) and combinatorial inverse semigroup is naturally semilattice-ordered.

Consider the set $B_2^1$ consisting of the following six zero-one $2\times 2$-matrices:
\begin{equation}
\label{eq:b21}
\begin{tabular}{cccccc}
$\left(\begin{matrix} 0&0\\0&0\end{matrix}\right)$,
&
$\left(\begin{matrix} 1&0\\0&0\end{matrix}\right)$,
&
$\left(\begin{matrix} 0&1\\0&0\end{matrix}\right)$,
&
$\left(\begin{matrix} 0&0\\1&0\end{matrix}\right)$,
&
$\left(\begin{matrix} 0&0\\0&1\end{matrix}\right)$,
&
$\left(\begin{matrix} 1&0\\0&1\end{matrix}\right)$.
\end{tabular}
\end{equation}
They form an inverse semigroup (even an inverse monoid) under the usual matrix multiplication and transposition. The inverse monoid $(B_2^1,\cdot,{}^{-1})$ is known as the 6-\emph{element Brandt monoid}. Our first result answers Problem~\ref{prob:jackson} for ``almost all'' finite combinatorial naturally semilattice-ordered inverse semigroups:

\begin{theorem}
\label{thm:comb-nfb}
If $(B_2^1,\cdot,{}^{-1})$ satisfies all identities of a finite combinatorial inverse semigroup $(S,\cdot,{}^{-1})$, then the \ais\ $(S,+_{\mathrm{nat}},\cdot)$ admits no finite identity basis.
\end{theorem}

\begin{remark}
\label{rm:almostall}
Two algebras of the same type that satisfy the same identities are called \emph{equationally equivalent}. If a finite combinatorial inverse semigroup $(S,\cdot,{}^{-1})$ satisfies an identity that fails in the $6$-element Brandt monoid, then either $|S|=1$ or $(S,\cdot,{}^{-1})$ is equationally equivalent to either the 2-element semilattice or the 5-\emph{element Brandt semigroup} $(B_2,\cdot,{}^{-1})$ where $B_2$ consists of the first five matrices in \eqref{eq:b21}; see \cite[Section XII.4]{Pet84}, in particular, Corollary XII.4.14 therein. From this, it readily follows that up to equational equivalence, Theorem~\ref{thm:comb-nfb} does not apply to  only two nontrivial \ais{}s $(S,+_{\mathrm{nat}},\cdot)$ coming from a finite combinatorial inverse semigroup: these two are $(Y_2,+_{\mathrm{nat}},\cdot)$ where $(Y_2,\cdot)$ is the 2-element semilattice and $(B_2,+_{\mathrm{nat}},\cdot)$. It is known and easy to verify that the \ais\ $(Y_2,+_{\mathrm{nat}},\cdot)$ is FB (in fact, the single identity $xy\approx x+y$ forms an identity basis for $(Y_2,+_{\mathrm{nat}},\cdot)$). Thus, the 5-element Brandt semigroup remains the only finite combinatorial inverse semigroup with yet unknown answer to the FBP for its derived \ais.
\end{remark}

Another important family of naturally semilattice-ordered inverse semigroups is related to the symmetric inverse monoids; see \cite[Section IV.1]{Pet84} or \cite[Chapter I]{Law99} for an explanation of the role played by these monoids in the theory of inverse semigroups. For a non-empty set $X$, let $I(X)$ stand for the set of all partial one-to-one transformations on $X$. The \emph{symmetric inverse monoid} on $X$ is $(I(X),\cdot,{}^{-1})$ where for all $\alpha,\beta\in I(X)$, the product $\alpha\beta$ is the usual composition of transformations and $\alpha^{-1}$ is the inverse transformation of $\alpha$. The natural partial order $\le_{\mathrm{nat}}$ on $(I(X),\cdot,{}^{-1})$ is nothing but the usual extension order of transformations: $\beta\in I(X)$ extends $\alpha\in I(X)$ if $\alpha(x)=\beta(x)$ for each $x\in X$ at which $\alpha(x)$ is defined. Clearly, $(I(X),\le_{\mathrm{nat}})$ is an inf-semilattice: for any $\alpha,\beta\in I(X)$, their infimum is the transformation $\gamma$ defined on the set $\{x\in X\mid\alpha(x)=\beta(x)\}$ by the rule $\gamma(x):=\alpha(x)$. Thus, we get the naturally semilattice-ordered inverse semigroup $(I(X),+_{\mathrm{nat}},\cdot)$.

If the set $X$ is finite with $t$ elements, the symmetric inverse monoid on $X$ can be conveniently identified with the \emph{rook monoid} $\mathcal{R}_t:=(R_t,\cdot,{}^{-1})$ where $R_t$ is the set of all zero-one $t\times t$-matrices with at most one entry equal to 1 in each row and column and the operations are the usual matrix multiplication and transposition. (The name `rook monoid' suggested by Solomon \cite{Sol02} refers to the fact that matrices in  $R_t$ encode placements of nonattacking rooks on a $t\times t$ chessboard.) In this model, the addition $+_{\mathrm{nat}}$ is nothing but the Hadamard (entrywise) product of matrices: $(a_{ij})_{t\times t}+_{\mathrm{nat}}(b_{ij})_{t\times t}=(a_{ij}b_{ij})_{t\times t}$. Our second result solves the FBP for the `rook semirings' $(R_t,+_{\mathrm{nat}},\cdot)$.

\begin{theorem}
\label{thm:rook}
The \ais\ $(R_t,+_{\mathrm{nat}},\cdot)$ admits a finite identity basis if and only if $t=1$.
\end{theorem}

\begin{remark}
From the proofs in Sect.~\ref{sec:proofs}, it will be clear that the results of Theorems~\ref{thm:comb-nfb} and \ref{thm:rook} hold also with the signature $\{+,\cdot,0\}$.
\end{remark}

We employ the usual scheme of ``semantic'' proofs for the absence of a finite identity basis: to prove that a given \ais\ $\mathcal{S}$ has no identity basis involving less than any fixed number $k$ of variables (and hence, no finite identity basis), one constructs for each $k$, an \ais\ $\mathcal{S}_k$ and an identity $\mathbf w_k\approx \mathbf w'_k$ of $\mathcal{S}$ such that $\mathcal{S}_k$ satisfies all identities of $\mathcal{S}$ with less than $k$ variables but refutes the identity $\mathbf w_k\approx\mathbf w'_k$.  In Sects.~\ref{sec:height} and~\ref{sec:kadourek} we prepare these ingredients of the proof, and in Sect.~\ref{sec:proofs} we put them together to prove Theorem~\ref{thm:main}, a general result that provides a large class of NFB \ais{}s. Theorems~\ref{thm:comb-nfb} and \ref{thm:rook} then follow easily.

We assume the reader's acquaintance with basic notions and results of the theory of inverse semigroups. All these can be found in the early chapters of the monographs \cite{Law99,Pet84}.

\section{The identities $\mathbf v_{n,m}^{(h)} \approx (\mathbf v_{n,m}^{(h)})^2$}
\label{sec:height}

Here we aim to show that every finite inverse semigroup with abelian subgroups satisfies a specific identity involving only multiplication. We construct the identity by climbing up a principal series of the semigroup, and the construction works fine for semigroups with finite principal series whose factors are either abelian groups of finite exponent or Brandt semigroups over such groups. Recall that a \emph{principal series} of a semigroup $(S,\cdot)$ is a chain
\begin{equation}
\label{eq:series}
S_0\subset S_1\subset \dots \subset S_h=S
\end{equation}
of ideals $S_j$ of $(S,\cdot)$ such that such that there is no ideal of $(S,\cdot)$ strictly below $S_0$ nor strictly between $S_{j-1}$ and $S_j$ for $j=1,\dots,h$. By the \emph{factors} of the principal series \eqref{eq:series} we mean the Rees quotients $(S_j/S_{j-1},\cdot)$, $j=1,\dots,h$. To keep the premises of further statements compact, the expression $(h,m)$-\emph{semigroup} is used for any semigroup $(S,\cdot)$ that has a principal series
\eqref{eq:series} in which $(S_0,\cdot)$ is an abelian group of exponent dividing $m$ and each factor $(S_j/S_{j-1},\cdot)$, $j=1,\dots,h$, is a Brandt semigroup over an abelian group of exponent dividing $m$. It is easy to see that an $(h,m)$-semigroup is necessarily inverse.

We start our ascent from $(1,m)$-semigroups with $S_0=\{0\}$, where 0 is the zero of $(S,\cdot)$. In this case, $(S,\cdot)$ is just a Brandt semigroup over an abelian group of exponent dividing $m$.
As we need some calculations in Brandt semigroups, we recall how they are defined.  Let $I$ be a non-empty set and let $\mathcal{G}=(G,\cdot)$ be a group. The Brandt semigroup $\mathcal{B}_{G,I}$ over $\mathcal{G}$ has $B_{G,I}:=I \times G \times I\cup \{0\}$ as its carrier set, and the multiplication in $\mathcal{B}_{G,I}$ is defined by
\begin{align*}
&(\ell_1,g_1,r_1)\cdot (\ell_2,g_2,r_2):=
\begin{cases}
(\ell_1,g_1g_2,r_2)&\text{if }r_1=\ell_2,\\
0&\text{otherwise,}
\end{cases}&&\text{for all $\ell_1,\ell_2,r_1,r_2\in I,\ g_1,g_2\in G$,}\\
&(\ell,g,r)\cdot 0=0\cdot(\ell,g,r)=0\cdot0:=0&&\text{for all $\ell,r\in I,\ g\in G$.}
\end{align*}

Now we are going to introduce a family of words $\mathbf u_{n,k,m}$ used as building blocks for our identities. For any $i\ge 1$, let $X_i^{(1)}=\{x_1,x_2,\dots,x_i\}$. Now for any $n,k\ge 0$ with $n+k>0$ and for any $m\ge 1$, we define the following word over $X_{n+k}^{(1)}$:
\[
\mathbf u_{n,k,m}:= x_1x_2\cdots x_{n+k}\,\Bigl(x_nx_{n-1}\cdots x_1\cdot x_{n+1}x_{n+2}\cdots x_{n+k}\Bigr)^{2m-1}.
\]
For clarity, we specify how the words $\mathbf u_{n,k,m}$ with $k=0,1$ or $n=0,1$ look like:
\begin{align}
\mathbf u_{n,0,m}&=x_1x_2\cdots x_n\,\Bigl(x_nx_{n-1}\cdots x_1\Bigr)^{2m-1},\label{eq:un0}\\
\mathbf u_{n,1,m}&=x_1x_2\cdots x_{n+1}\,\Bigl(x_nx_{n-1}\cdots x_1x_{n+1}\Bigr)^{2m-1},\label{eq:un1}\\
\mathbf u_{0,k,m}&=\Bigl(x_1x_2\cdots x_k\Bigr)^{2m},\label{eq:u0k}\\
\mathbf u_{1,k,m}&=\Bigl(x_1x_2\cdots x_{k+1}\Bigr)^{2m}.\label{eq:u1k}
\end{align}

The next result could have been deduced from a known characterization of semigroup identities of Brandt semigroups~\cite{Mash79}; see also \cite{Rei08}, but we verify it by a direct computation.

\begin{lemma}
\label{lem:id-BI}
For any non-empty set $I$ and any abelian group $\mathcal{G}=(G,\cdot)$ of exponent dividing $m$, the semigroup $\mathcal{B}_{G,I}$ satisfies all identities $\mathbf u_{n,k,m} \approx \mathbf u_{n,k,m}^2$ with $n,k\ge 0$ and $n+k>0$.
\end{lemma}

\begin{proof}
If $n\in\{0,1\}$, the claim holds because $\mathbf u_{n,k,m}$ is of the form either~\eqref{eq:u0k} or~\eqref{eq:u1k} and it is known (and easy to verify) that $\mathcal{B}_{G,I}$ satisfies the identity $x^2 \approx x^{2+m}$.

Let $n\ge2$ and consider an arbitrary substitution $\tau\colon X_{n+k}^{(1)}\to B_{G,I}$. We aim to show that $\tau(\mathbf u_{n,k,m})$ is an idempotent. If $\tau(\mathbf u_{n,k,m})=0$, there is nothing to prove.
Thus, for the rest of the proof we assume that $\tau(\mathbf u_{n,k,m})\ne 0$. Then all values of the substitution $\tau$ lie in the set $I \times G \times I$ of non-zero elements of $\mathcal{B}_{G,I}$. For each $i\in\{1,2,\dots,n+k\}$, let $\tau(x_i)=(\ell_i,g_i,r_i)$, where $\ell_i,r_i\in I$ and $g_i\in G$. Denote by $e$ the identity element of the group $\mathcal{G}$.

First, consider the case $k=0$. If $\tau(\mathbf u_{n,0,m})\ne 0$, then due to \eqref{eq:un0}, we have
\begin{align*}
\tau(\mathbf u_{n,0,m})&=(\ell_{1},g_1^{2m}g_2^{2m}\cdots g_n^{2m},r_{1})&&\text{since $\mathcal{G}$ is abelian,}\\
                       &=(\ell_{1},e,r_{1})&&\text{since the exponent of $\mathcal{G}$ divides $m$.}
\end{align*}
We verify that $\ell_j=r_j$ for all $j=n,n-1,\dots,1$ by backwards induction. Since $x_n^2$ occurs as a factor in the word $\mathbf u_{n,0,m}$, we must have $\tau(x_n^2)\ne 0$, whence $\ell_n=r_n$. If $j>1$, both $x_{j-1}x_j$ and $x_jx_{j-1}$ occur as factors in $\mathbf u_{n,0,m}$. We then have
\begin{align*}
r_{j-1}&=\ell_j&&\text{since $\tau(x_{j-1}x_j)\ne 0$,}\\
       &=r_j&&\text{by the induction assumption,}\\
       &=\ell_{j-1}&&\text{since $\tau(x_jx_{j-1})\ne 0$.}
\end{align*}
Since $\ell_1=r_1$, we have $(\ell_{1},e,r_{1})^2=(\ell_{1},e,r_{1})$, that is,  $\tau(\mathbf u_{n,0,m})$ is an idempotent.

Now consider the case $k=1$. If $\tau(\mathbf u_{n,1,m})\ne 0$, then due to \eqref{eq:un1}, we have
\begin{align*}
\tau(\mathbf u_{n,1,m})&=(\ell_{1},g_1^{2m}g_2^{2m}\cdots g_n^{2m},r_{n+1})&&\text{since $\mathcal{G}$ is abelian,}\\
                       &=(\ell_{1},e,r_{n+1})&&\text{since the exponent of $\mathcal{G}$ divides $m$.}
\end{align*}
Since $\tau(x_1x_2\cdots x_{n+1}x_n)\ne 0$, we have
\[
r_{1}=\ell_{2},\,r_{2}=\ell_{3},\dots,r_{n-1}=\ell_{n},\,r_{n}=\ell_{n+1},\,r_{n+1}=\ell_{n}.
\]
Since $\tau(x_nx_{n-1}\cdots x_1x_{n+1})\ne 0$, we also have
\[
r_{n}=\ell_{n-1},\,r_{n-1}=\ell_{n-2},\dots,r_{2}=\ell_{1},\,r_{1}=\ell_{n+1}.
\]
Therefore, we obtain
\[
r_{n+1}=\ell_{n}=r_{n-1}=\ell_{n-2}=r_{n-3}=\dots=\begin{cases}
\ell_{1}&\text{if $n$ is odd},\\
r_{1}&\text{if $n$ is even}.
\end{cases}
\]
Besides that, if $n$ is even, then
\[
r_{1}=\ell_{n+1}=r_{n}=\ell_{n-1}=r_{n-2}=\ell_{n-3}=\dots=\ell_{1}.
\]
We see that $r_{n+1}=\ell_{1}$ in either case. Therefore, $(\ell_{1},e,r_{n+1})^2=(\ell_{1},e,r_{n+1})$, that is, $\tau(\mathbf u_{n,1,m})$ is an idempotent.

Finally, substituting the word $x_{n+1}x_{n+2}\cdots x_{n+k}$ for the variable $x_{n+1}$ in $\mathbf u_{n,1,m} \approx \mathbf u_{n,1,m}^2$ yields the identity $\mathbf u_{n,k,m} \approx \mathbf u_{n,k,m}^2$. Hence  the latter identity also holds in $\mathcal{B}_{G,I}$.
\end{proof}

For any $n,m\ge 1$, let
\begin{equation}
\label{eq:v1}
\mathbf v_{n,m}^{(1)}:=\mathbf u_{n,n,m}=x_1x_2\cdots x_{2n}\,\Bigl(x_nx_{n-1}\cdots x_1\cdot x_{n+1}x_{n+2}\cdots x_{2n}\Bigr)^{2m-1}.
\end{equation}
Observe that in $\mathbf v_{n,m}^{(1)}$, each variable occurs $2m$ times.

For any semigroup $(S,\cdot)$, let $E(S)$ stand for the set of all its idempotents. Recall that if the semigroup is inverse, then the set $E(S)$ is closed under multiplication.

\begin{lemma}
\label{lem:id-E(S)cupD}
Let $(S,\cdot)$ be an $(h,m)$-semigroup and \eqref{eq:series} its principal series. If $S_0=\{0\}$ and $S=E(S)\cup S_1$, then $(S,\cdot)$ satisfies the identity $\mathbf v_{n,m}^{(1)} \approx (\mathbf v_{n,m}^{(1)})^2$.
\end{lemma}

\begin{proof}
Since $S_0=\{0\}$, the semigroup $(S_1,\cdot)$ is the Brandt semigroup $\mathcal{B}_{G,I}$ for some abelian group $\mathcal{G}=(G,\cdot)$ of exponent dividing $m$ and some non-empty set $I$. Clearly, each non-zero idempotent in $\mathcal{B}_{G,I}$ is of the form $(i,e,i)$ where $i\in I$ and $e$ is the identity element of the group $\mathcal{G}$. For any idempotent $f\in E(S)$, the product $f(i,e,i)$ is an idempotent in $\mathcal{B}_{G,I}$. If this product is not 0, then $f(i,e,i)=(j,e,j)$ for some $j\in I$. Multiplying the equality through by $(j,e,j)$ on the right yields $f(i,e,i)(j,e,j)=(j,e,j)$ whence $j=i$. We conclude that for all $f\in E(S)$ and $i\in I$, either $f(i,e,i)=(i,e,i)$ or $f(i,e,i)=0$. Now take an arbitrary element $d=(\ell,g,r)\in I\times G\times I$. Then $d=(\ell,e,\ell)d$ whence $fd=f(\ell,e,\ell)d$ for any idempotent $f\in E(S)$. We see that either $fd=d$ or $fd=0$, and dually, either $df=d$ or $df=0$.

To prove the lemma, we have to verify that $\tau(\mathbf v_{n,m}^{(1)})$ is an idempotent for an arbitrary substitution $\tau\colon X_{2n}^{(1)}\to S$. If $\tau(x_k)\in E(S)$ for all $k\in\{1,2,\dots,2n\}$, then $\tau(\mathbf v_{n,m}^{(1)})\in E(S)$ because $E(S)$ is closed under multiplication. Otherwise let $\{k_1,k_2,\dots,k_{p+q}\}$ with
\[
1\le k_1<k_2<\dots<k_p\le n< k_{p+1}<k_{p+2}<\dots<k_{p+q}\le 2n
\]
be the set of all indices $k$ such that $\tau(x_k)\notin E(S)$. (Here $p=0$ or $q=0$ is possible but $p+q>0$.) Since $S=E(S)\cup S_1$, we have $\tau(x_{k_1}),\tau(x_{k_2}),\dots,\tau(x_{k_{p+q}})\in S_1$. The argument in the preceding paragraph implies that either $\tau(\mathbf v_{n,m}^{(1)})=0$ or removing all $\tau(x_k)$ such that $\tau(x_k)\in E(S)$ does not change the value of $\tau(\mathbf v_{n,m}^{(1)})$. In the former case, the claim holds, and in the latter case, consider the substitution $\tau'\colon X_{p+q}^{(1)}\to S_1$ given by $\tau'(x_t):=\tau(x_{k_t})$ for all $s\in\{1,2,\dots,p+q\}$. Then $\tau(\mathbf v_{n,m}^{(1)})=\tau'(\mathbf u_{p,q,m})$. By Lemma~\ref{lem:id-BI} the Brandt semigroup $(S_1,\cdot)$ satisfies $\mathbf u_{p,q,m} \approx \mathbf u_{p,q,m}^2$ whence $\tau(\mathbf v_{n,m}^{(1)})$ is an idempotent also in this case.
\end{proof}

We proceed  with constructing identities holding in arbitrary $(h,m)$-semigroups. For any $i,h\ge1$, let
\[
X_i^{(h)}:=\{x_{i_1i_2\cdots i_h}\mid i_1,i_2,\dots,i_h\in\{1,2,\dots,i\}\}.
\]
For any $n,m,h\ge 1$, we introduce words $\mathbf v_{n,m}^{(h)}$ over $X_{2n}^{(h)}$ by induction on $h$. The word $\mathbf v_{n,m}^{(1)}$ over $X_{2n}^{(1)}$ has been defined in \eqref{eq:v1}. Then, assuming that $h>1$ and the word $\mathbf v_{n,m}^{(h-1)}$ over $X_{2n}^{(h-1)}$ has already been defined, we create $2n$ copies of this word over the alphabet $X_{2n}^{(h)}$ as follows. We start by taking for every $j\in\{1,2,\dots,2n\}$, the substitution
\[
\sigma_{2n,j}^{(h)}\colon X_{2n}^{(h-1)}\to X_{2n}^{(h)}
\]
that appends $j$ to the indices of its arguments, that is,
\[
\sigma_{2n,j}^{(h)}(x_{i_1i_2\dots i_{h-1}}):= x_{i_1i_2\dots i_{h-1}j} \text{ for all } i_1,i_2,\dots,i_{h-1}\in\{1,2,\dots,2n\}.
\]
Then we let $\mathbf v_{n,m,j}^{(h-1)}:=\sigma_{2n,j}^{(h)}(\mathbf v_{n,m}^{(h-1)})$ and define
\begin{equation}
\label{eq:vnh}
\mathbf v_{n,m}^{(h)}:=\mathbf v_{n,m,1}^{(h-1)}\cdots\mathbf v_{n,m,2n}^{(h-1)}\Bigl(\mathbf v_{n,m,n}^{(h-1)}\cdots\mathbf v_{n,m,1}^{(h-1)}\cdot \mathbf v_{n,m,n+1}^{(h-1)}\cdots\mathbf v_{n,m,2n}^{(h-1)}\Bigr)^{2m-1}.
\end{equation}

Comparing the definitions \eqref{eq:v1} and \eqref{eq:vnh}, one readily sees that the word $\mathbf v_{n,m}^{(h)}$ is nothing but the image of $\mathbf v_{n,m}^{(1)}$ under the substitution $x_j \mapsto \mathbf v_{n,m,j}^{(h-1)}$, $j\in\{1,2,\dots,2n\}$.

\begin{proposition}
\label{prop:id-fh}
Let $(S,\cdot)$ be an $(h,m)$-semigroup with principal series \eqref{eq:series} and $S_0=\{0\}$. For any $n\ge 2$, $(S,\cdot)$ satisfies the identity
\begin{equation}
\label{eq:IDh}
\mathbf v_{n,m}^{(h)} \approx (\mathbf v_{n,m}^{(h)})^2.
\end{equation}
\end{proposition}

\begin{proof}
We induct on $h$. If $h=1$, then $(S,\cdot)$ is a Brandt semigroup over an abelian group of exponent dividing $m$. Lemma~\ref{lem:id-BI} implies that $\mathbf v_{n,m}^{(1)} \approx (\mathbf v_{n,m}^{(1)})^2$ holds in $(S,\cdot)$ for all $n\ge2$.

Let $h>1$. The Rees quotient $(S/S_1,\cdot)$ is an $(h-1,m)$-semigroup whose principal series starts with the zero term. By the induction assumption, $(S/S_1,\cdot)$ satisfies the identity $\mathbf v_{n,m}^{(h-1)} \approx (\mathbf v_{n,m}^{(h-1)})^2$ for any $n\ge2$. This readily implies that any substitution $X_{2n}^{(h-1)}\to S$ sends the word $\mathbf v_{n,m}^{(h-1)}$ to either an idempotent in $S\setminus S_1$ or an element in $S_1$. By the construction, all words of the form $\mathbf v_{n,m,j}^{(h-1)}$ are obtained from the word $\mathbf v_{n,m}^{(h-1)}$ by renaming its variables. Therefore, for every substitution $\tau\colon X_{2n}^{(h)}\to S$, the elements $\tau(\mathbf v_{n,m,1}^{(h-1)})$, $\tau(\mathbf v_{n,m,2}^{(h-1)})$, \dots, $\tau(\mathbf v_{n,m,2n}^{(h-1)})$ lie in either $E(S)$ or $S_1$. The subsemigroup $(E(S)\cup S_1,\cdot)$ satisfies the identity $\mathbf v_{n,m}^{(1)} \approx (\mathbf v_{n,m}^{(1)})^2$ by Lemma~\ref{lem:id-E(S)cupD}. This, together with the observation made after the equality~\eqref{eq:vnh}, implies that $\tau(\mathbf v_{n,m}^{(h)})=\tau((\mathbf v_{n,m}^{(h)})^2)$. Since the substitution $\tau$ is arbitrary, the semigroup $(S,\cdot)$ satisfies \eqref{eq:IDh}.
\end{proof}

Now we remove the restriction $S_0=\{0\}$.

\begin{proposition}
\label{prop:id-fh+1}
The identity $\mathbf v_{n,m}^{(h+1)} \approx (\mathbf v_{n,m}^{(h+1)})^2$ with $n\ge 2$ holds in each $(h,m)$-semigroup.
\end{proposition}

\begin{proof}
Let $(S,\cdot)$ be an $(h,m)$-semigroup with principal series \eqref{eq:series}. Consider the Rees quotient $(S/S_0,\cdot)$. This is an $(h,m)$-semigroup whose principal series starts with the zero term. By Proposition \ref{prop:id-fh}, the semigroup $(S/S_0,\cdot)$ satisfies \eqref{eq:IDh}. Since $\mathbf v_{n,m,j}^{(h)}=\sigma_{2n,j}(\mathbf v_{n,m}^{(h)})$ for all $j=1,2,\dots,2n$, this implies that for every substitution $\tau\colon X_{2n}^{(h+1)}\to S$, the element $\tau(\mathbf v_{n,m,j}^{(h)})$ represents an idempotent of $(S/S_0,\cdot)$. Therefore, $\tau(\mathbf v_{n,m,j}^{(h)})$ lies in either $E(S)\setminus S_0$ or $S_0$. Substituting $h+1$ for $h$ in \eqref{eq:vnh}, we see that
\[
\mathbf v_{n,m}^{(h+1)}=\mathbf v_{n,m,1}^{(h)}\cdots\mathbf v_{n,m,2n}^{(h)}\Bigl(\mathbf v_{n,m,n}^{(h)}\cdots\mathbf v_{n,m,1}^{(h)}\cdot \mathbf v_{n,m,n+1}^{(h)}\cdots\mathbf v_{n,m,2n}^{(h)}\Bigr)^{2m-1}.
\]
Recall that by the definition of an $(h,m)$-semigroup, $(S_0,\cdot)$ is an abelian group of exponent dividing $m$. Denote by $e$ the identity element of the group. For each $f\in E(S)$, the product $fe$ is an idempotent in $S_0$ whence $fe=e$ since a group has no idempotent except its identity element. Consequently, for every $g\in S_0$, we have $fg=feg=eg=g$. Dually, $gf=g$ for all $g\in S_0$, $f\in E(S)$. Now consider the substitution $\overline{\tau}\colon X_{2n}^{(1)}\to S_0$ defined by
\[
\overline{\tau}(x_j):=\begin{cases}
\tau(\mathbf v_{n,m,j}^{(h)})&\text{if }\tau(\mathbf v_{n,m,j}^{(h)})\text{ belongs to } S_0,\\
e&\text{otherwise},
\end{cases}\quad\text{for each $j\in\{1,2,\dots,2n\}$.}
\]
As we know that $\tau(\mathbf v_{n,m,j}^{(h)})\in E(S)\cup S_0$ for all $j\in\{1,2,\dots,2n\}$ and $fg=gf=g$ for all $g\in S_0$, $f\in E(S)$, we conclude that $\tau(\mathbf v_{n,m}^{(h+1)})=\overline{\tau}(\mathbf v_{n,m}^{(1)})$. Since $(S_0,\cdot)$ is an abelian group of exponent dividing $m$, it satisfies the identity $\mathbf v_{n,m}^{(1)} \approx 1$. Hence $\tau(\mathbf v_{n,m}^{(h+1)})=\overline{\tau}(\mathbf v_{n,m}^{(1)})=e$. Since the substitution $\tau$ is arbitrary, $(S,\cdot)$ satisfies the identity $\mathbf v_{n,m}^{(h+1)} \approx (\mathbf v_{n,m}^{(h+1)})^2$.
\end{proof}

We conclude this section with proving that the rook monoids $\mathcal{R}_2$ and $\mathcal{R}_3$ satisfy certain identities of the form~\eqref{eq:IDh}.
\begin{proposition}
\label{prop:rook2and3}
\emph{(1)} The rook monoid $\mathcal{R}_2$ satisfies the identity $\mathbf v_{n,2}^{(2)} \approx (\mathbf v_{n,2}^{(2)})^2$ for any $n\ge 2$.

\emph{(2)} The rook monoid $\mathcal{R}_3$ satisfies the identity $\mathbf v_{n,6}^{(4)} \approx (\mathbf v_{n,6}^{(4)})^2$ for any $n\ge 2$.
\end{proposition}

\begin{proof}
The rook monoid $\mathcal{R}_t$ has a principal series
\[
\{0\}=I_0\subset I_1\subset \dots \subset I_t=R_t
\]
such that for each $k=1,\dots,s$, the Rees factor $(I_k/I_{k-1},\cdot)$ is a Brandt semigroup over the symmetric group $\mathrm{Sym}_k$, that is, the group of all permutations of $k$ symbols; see, e.g., \cite[Section 2]{Munn57}.
Since $\mathrm{Sym}_1$ is trivial and $\mathrm{Sym}_2$ consists of two elements, we see that $\mathcal{R}_2$ is a (2,2)-semi\-group. Hence Proposition~\ref{prop:id-fh} applies, yielding that $\mathbf v_{n,2}^{(2)} \approx (\mathbf v_{n,2}^{(2)})^2$ holds in $\mathcal{R}_2$ for all $n\ge 2$.

The group $\mathrm{Sym}_3$ is non-abelian, and therefore,  Proposition~\ref{prop:id-fh} does not apply to $\mathcal{R}_3$. However, it does apply to the subsemigroup $\mathcal{R}'_3:=(R'_3,\cdot,{}^{-1})$ where the set $R'_3$ is obtained from $R_3$ by removing the three transposition matrices
\begin{equation}\label{eq:transpositions}
\begin{pmatrix}
0&1&0\\
1&0&0\\
0&0&1
\end{pmatrix},\quad
\begin{pmatrix}
0&0&1\\
0&1&0\\
1&0&0
\end{pmatrix},\quad
\begin{pmatrix}
1&0&0\\
0&0&1\\
0&1&0
\end{pmatrix}.
\end{equation}
The subsemigroup $\mathcal{R}'_3$ has the principal series $\{0\}=I_0\subset I_1\subset I_2 \subset I'_3=R'_3$. All subgroups of $\mathcal{R}'_3$ have one, two, or three elements, and so, they all are abelian of exponent dividing 6. By Proposition~\ref{prop:id-fh} $\mathcal{R}'_3$ satisfies the identity $\mathbf v_{n,6}^{(3)} \approx (\mathbf v_{n,6}^{(3)})^2$.

It is easy to see that the word $\mathbf v_{n,6}^{(4)}$ is the image of the word $\mathbf v_{n,6}^{(3)}$ under a substitution $\zeta$ that sends every variable from $X^{(3)}_{2n}$ to a word obtained from $\mathbf v_{n,6}^{(1)}$ by renaming its variables. Indeed, in terms of the substitutions $\sigma_{2n,j}^{(h)}$ used in the definition of the words $\mathbf v_{n,m}^{(h)}$, one can express $\zeta$ as follows:
\[
\zeta(x_{i_1i_2i_3}):=\sigma_{2n,i_3}^{(4)}(\sigma_{2n,i_2}^{(3)}(\sigma_{2n,i_1}^{(2)}(\mathbf v_{n,6}^{(1)})))\ \text{ for all } i_1,i_2,i_3\in\{1,2,\dots,2n\}.
\]
Simply put, $\zeta(x_{i_1i_2i_3})$ is obtained by appending $i_1i_2i_3$ to the indices of all variables of $\mathbf v_{n,6}^{(1)}$.

Take an arbitrary substitution $\tau\colon X_{2n}^{(4)}\to R_3$. For any fixed $i_1,i_2,i_3\in\{1,2,\dots,2n\}$,  consider the substitution $\tau_{i_1i_2i_3}\colon X_{2n}^{(1)}\to R_3$ induced by $\tau$ via the rule $\tau_{i_1i_2i_3}(x_j):=\tau(x_{ji_1i_2i_3})$ for all $j\in \{1,2,\dots,2n\}$. Then we have
\[
\tau(\zeta(x_{i_1i_2i_3}))=\tau_{i_1i_2i_3}(\mathbf v_{n,6}^{(1)})\ \text{ for all } i_1,i_2,i_3\in\{1,2,\dots,2n\}.
\]
Thus, evaluating $\tau$ at the word $\mathbf v_{n,6}^{(4)}=\zeta(\mathbf v_{n,6}^{(3)})$ produces the same result as substitution of the elements $\tau_{i_1i_2i_3}(\mathbf v_{n,6}^{(1)})$ for the variables $x_{i_1i_2i_3}$ into the word $\mathbf v_{n,6}^{(3)}$. If $\tau_{i_1i_2i_3}(\mathbf v_{n,6}^{(1)})\in R'_3$ for all substitutions $\tau_{i_1i_2i_3}$, we can use the fact that $\mathcal{R}'_3$ satisfies the identity $\mathbf v_{n,6}^{(3)} \approx (\mathbf v_{n,6}^{(3)})^2$ as registered above and conclude that $\tau(\mathbf v_{n,6}^{(4)})=\tau((\mathbf v_{n,6}^{(4)})^2)$.

We see that it remains to show that $\rho(\mathbf v_{n,6}^{(1)})\in R'_3$ for every substitution $\rho\colon X_{2n}^{(1)}\to R_3$. If $\rho(x_j)\in I_2$ for some $j$, then $\rho(\mathbf v_{n,6}^{(1)})\in I_2\subset R'_3$ since $I_2$ is an ideal in $\mathcal{R}_3$. Assume that $\rho(x_j)\notin I_2$ for all $j\in\{1,2,\dots,2n\}$. The set $R_3\setminus I_2$ consists of six permutation zero-one $3\times 3$-matrices of which the three transposition matrices in~\eqref{eq:transpositions} have determinant $-1$ while the three other matrices have determinant 1 and belong to $R'_3$. Since each variable occurs in $\mathbf v_{n,6}^{(1)}$ an even number of times, $\rho(\mathbf v_{n,6}^{(1)})$ is a product of matrices with determinant $\pm1$ that has an even number of factors with determinant $-1$. Hence, $\rho(\mathbf v_{n,6}^{(1)})$ a permutation matrix with determinant 1, and therefore, $\rho(\mathbf v_{n,6}^{(1)})\in R'_3$.
\end{proof}

\section{The semigroups $S_n^{(h)}$}
\label{sec:kadourek}

We make use of a family of inverse semigroups constructed by Ka\softd{}ourek in~\cite[Section~2]{Kad03}. For the reader's convenience we reproduce Ka\softd{}ourek's construction here.

First, for all $n,h\ge 1$, define terms $\mathbf w_n^{(h)}$ of the signature $(\cdot,{}^{-1})$ over the alphabet $X_n^{(h)}=\{x_{i_1i_2\cdots i_h}\mid i_1,i_2,\dots,i_h\in\{1,2,\dots,n\}\}$ by induction on $h$. Put
\[
\mathbf w_n^{(1)}:=x_1x_2\cdots x_n\,x_1^{-1}x_2^{-1}\cdots x_n^{-1}.
\]
Then, assuming that, for any $h>1$, the unary term $\mathbf w_n^{(h-1)}$ over the alphabet $X_n^{(h-1)}$ has already been defined, we create $n$ copies of this term over the alphabet $X_n^{(h)}$ as follows. For every $j\in\{1,2,\dots,n\}$, we put
\[
\mathbf w_{n,j}^{(h-1)}:=\sigma_{n,j}^{(h)}(\mathbf w_n^{(h-1)}),
\]
where the substitution $\sigma_{n,j}^{(h)}\colon X_{n}^{(h-1)}\to X_{n}^{(h)}$ appends $j$ to the indices of its arguments, i.e.,
\[
\sigma_{n,j}^{(h)}(x_{i_1i_2\dots i_{h-1}}):= x_{i_1i_2\dots i_{h-1}j} \text{ for all } i_1,i_2,\dots,i_{h-1}\in\{1,2,\dots,n\}.
\]
Then we put
\[
\mathbf w_n^{(h)}:=\mathbf w_{n,1}^{(h-1)}\mathbf w_{n,2}^{(h-1)}\cdots \mathbf w_{n,n}^{(h-1)}(\mathbf w_{n,1}^{(h-1)})^{-1}(\mathbf w_{n,2}^{(h-1)})^{-1}\cdots (\mathbf w_{n,n}^{(h-1)})^{-1}.
\]
Define the length of unary terms over $X_n^{(h)}$ by letting
\[
|x_{i_1i_2\cdots i_h}|:=1\ \text{ for $x_{i_1i_2\cdots i_h}\in X_n^{(h)}$ and }\ |\mathbf{ww'}|:=|\mathbf w|+|\mathbf w'|,\  |(\mathbf w)^{-1}|:=|\mathbf w|\ \text{for all terms $\mathbf w,\mathbf w'$}.
\]
Then one has $|\mathbf w_n^{(h)}|=2^hn^h$. Also observe that for every $i_1,i_2,\dots,i_h\in\{1,2,\dots,n\}$, the term $\mathbf w_n^{(h)}$ has exactly $2^{h-1}$ occurrences of $x_{i_1i_2\dots i_h}$ and exactly $2^{h-1}$ occurrences of $x_{i_1i_2\dots i_h}^{-1}$.

Now for any $n\ge2$ and $h\ge1$, let $(S_n^{(h)},\cdot,{}^{-1})$ stand for the inverse semigroup of partial one-to-one transformations on the set $\{0,1,\dots,2^hn^h\}$ generated by $n^h$ transformations $\chi_{i_1i_2\dots i_h}$ with arbitrary indices $i_1,i_2,\dots, i_h\in\{1,2,\dots,n\}$ defined as follows:
\begin{itemize}
\item $\chi_{i_1i_2\dots i_h}(q-1)=q$ if and only if the element on the $q$th position in $\mathbf w_n^{(h)}$ from the left is $x_{i_1i_2\dots i_h}$;
\item $\chi_{i_1i_2\dots i_h}(q)=q-1$ if and only if the element on the $q$th position in $\mathbf w_n^{(h)}$ from the left is $x_{i_1i_2\dots i_h}^{-1}$.
\end{itemize}
Clearly, $(S_n^{(h)},\cdot,{}^{-1})$ is finite (as a collection of transformations on a finite set) and has 0 (the nowhere defined transformation).

For an illustration, consider the case $n=2$, $h=2$. Then
\[
\mathbf w_2^{(2)}=\underbrace{x_{11}x_{21}x_{11}^{-1}x_{21}^{-1}}_{\mathbf w_{2,1}^{(1)}}\ \underbrace{x_{12}x_{22}x_{12}^{-1}x_{22}^{-1}}_{\mathbf w_{2,2}^{(1)}}\ \underbrace{x_{21}x_{11}x_{21}^{-1}x_{11}^{-1}}_{(\mathbf w_{2,1}^{(1)})^{-1}}\ \underbrace{x_{22}x_{12}x_{22}^{-1}x_{12}^{-1}}_{(\mathbf w_{2,2}^{(1)})^{-1}}.
\]
The action of the generators of the inverse semigroup $(S_2^{(2)},\cdot,{}^{-1})$ are shown in Figure~\ref{fig:kadourek}.
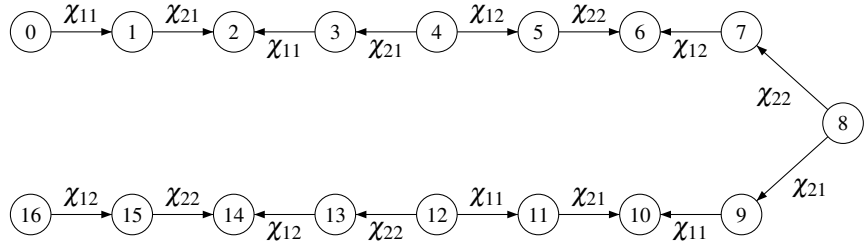
\begin{figure}[h]
\begin{center}
\unitlength 1.35mm
\begin{picture}(80,25)(0,-2.5)
\gasset{Nw=4,Nh=4,Nmr=2}
\node(u0)(0,18){\footnotesize 0}
\node(u1)(10,18){\footnotesize 1}
\node(u2)(20,18){\footnotesize 2}
\node(u3)(30,18){\footnotesize 3}
\node(u4)(40,18){\footnotesize 4}
\node(u5)(50,18){\footnotesize 5}
\node(u6)(60,18){\footnotesize 6}
\node(u7)(70,18){\footnotesize 7}
\node(u8)(80,9){\footnotesize 8}
\node(u9)(70,0){\footnotesize 9}
\node(u10)(60,0){\footnotesize 10}
\node(u11)(50,0){\footnotesize 11}
\node(u12)(40,0){\footnotesize 12}
\node(u13)(30,0){\footnotesize 13}
\node(u14)(20,0){\footnotesize 14}
\node(u15)(10,0){\footnotesize 15}
\node(u16)(0,0){\footnotesize 16}
\drawedge(u0,u1){$\chi_{11}$}
\drawedge(u1,u2){$\chi_{21}$}
\drawedge(u3,u2){$\chi_{11}$}
\drawedge(u4,u3){$\chi_{21}$}
\drawedge(u4,u5){$\chi_{12}$}
\drawedge(u5,u6){$\chi_{22}$}
\drawedge(u7,u6){$\chi_{12}$}
\drawedge(u8,u7){$\chi_{22}$}
\drawedge(u8,u9){$\chi_{21}$}
\drawedge(u9,u10){$\chi_{11}$}
\drawedge(u11,u10){$\chi_{21}$}
\drawedge(u12,u11){$\chi_{11}$}
\drawedge(u12,u13){$\chi_{22}$}
\drawedge(u13,u14){$\chi_{12}$}
\drawedge(u15,u14){$\chi_{22}$}
\drawedge(u16,u15){$\chi_{12}$}
\end{picture}
\caption{The generators of the inverse semigroup $(S_2^{(2)},\cdot,{}^{-1})$}\label{fig:kadourek}
\end{center}
\end{figure}

\noindent Observe that the partial transformation that one gets from the term $\mathbf w_2^{(2)}$ by evaluating each variable $x_{ij}$ at the transformation $\chi_{ij}^{-1}$ maps 16 to 0 and is undefined elsewhere.

\medskip

We need two properties of the inverse semigroups $(S_n^{(h)},\cdot,{}^{-1})$. The first one was deduced in \cite{Kad03} from an effective membership test for the inverse semigroup variety generated by the 6-element Brandt monoid $(B_2^1,\cdot,{}^{-1})$ that had been devised in~\cite{Kad91}.

\begin{proposition}[\!{\mdseries\cite[Corollary 3.2]{Kad03}}]
\label{prop:kadourek1}
Let $n\ge 2$ and $h\ge 1$. Any inverse subsemigroup of $(S_n^{(h)},\cdot,{}^{-1})$ generated by less than $n$ elements satisfies all identities of the $6$-element Brandt monoid $(B_2^1,\cdot,{}^{-1})$.
\end{proposition}

Since $(B_2^1,\cdot,{}^{-1})$ satisfies the identity $x^2\approx x^3$, applying Proposition~\ref{prop:kadourek1} to monogenic inverse subsemigroups of $(S_n^{(h)},\cdot,{}^{-1})$ yields the following fact:

\begin{corollary}
\label{cor:x2x3}
The identity $x^2\approx x^3$ holds in $(S_n^{(h)},\cdot,{}^{-1})$ for all $n\ge 2$ and $h\ge 1$.
\end{corollary}

The second property of the inverse semigroup $(S_n^{(h)},\cdot,{}^{-1})$ we need deals with its multiplicative reduct $(S_n^{(h)},\cdot)$ and appears to be new.

\begin{proposition}
\label{prop:Sn(h)}
Let $n\ge 2$ and $h,m\ge 1$. The semigroup $(S_n^{(h)},\cdot)$ violates the identity \eqref{eq:IDh}.
\end{proposition}

\begin{proof}
For any alphabet $X$, let $\overline{X}$ stand for the union of $X$ with the set $\{x^{-1}\mid x\in X\}$ of formal inverses of variables in $X$. We construct substitutions $\varphi_n^{(h)}$ and $\psi_n^{(h)}$ from $X_{2n}^{(h)}$ onto $\overline{X}_n^{(h)}$ such that $\varphi_n^{(h)}(\mathbf v_{n,m}^{(h)})\approx \mathbf w_n^{(h)}$ and $\psi_n^{(h)}(\mathbf v_{n,m}^{(h)})\approx (\mathbf w_n^{(h)})^{-1}$ in every inverse semigroup.

We induct on $h$. If $h=1$, we let
\[
\varphi_n^{(1)}(x_i):=
\begin{cases}
x_i&\text{if }1\le i\le n,\\
x_{i-n}^{-1}&\text{if }n+1\le i\le 2n,
\end{cases}
\quad\text{and}\quad\psi_n^{(1)}(x_i):=
\begin{cases}
x_{n+1-i}&\text{if }1\le i\le n,\\
x_{2n+1-i}^{-1}&\text{if }n+1\le i\le 2n.
\end{cases}
\]
Then
\[
\varphi_n^{(1)}(\mathbf v_{n,m}^{(1)})=x_1x_2\cdots x_n\,x_1^{-1}x_2^{-1}\cdots x_n^{-1}\,(x_nx_{n-1}\cdots x_1\,x_1^{-1}x_2^{-1}\cdots x_n^{-1})^{2m-1}.
\]
Since inverse semigroups satisfy $(x_nx_{n-1}\cdots x_1)^{-1}\approx x_1^{-1}x_2^{-1}\cdots x_n^{-1}$ and $x^{-1}xx^{-1}\approx x^{-1}$, we conclude that $\varphi_n^{(1)}(\mathbf v_{n,m}^{(1)})\approx x_1x_2\cdots x_n\,x_1^{-1}x_2^{-1}\cdots x_n^{-1}=\mathbf w_n^{(1)}$ in every inverse semigroup. Similarly, we get that in every inverse semigroup,
\begin{align*}
\psi_n^{(1)}(\mathbf v_{n,m}^{(1)})&=x_nx_{n-1}\cdots x_1\,x_n^{-1}x_{n-1}^{-1}\cdots x_1^{-1}\,(x_1x_2\cdots x_n\,x_n^{-1}x_{n-1}^{-1}\cdots x_1^{-1})^{2m-1}\\
&\approx x_nx_{n-1}\cdots x_1\,x_n^{-1}x_{n-1}^{-1}\cdots x_1^{-1}=(\mathbf w_n^{(h)})^{-1}.
\end{align*}

Let $h>1$. By the induction assumption, there are substitutions $\varphi_n^{(h-1)}$ and $\psi_n^{(h-1)}$ from $X_{2n}^{(h-1)}$ onto $\overline{X}_n^{(h-1)}$ such that all inverse semigroups satisfy $\varphi_n^{(h-1)}(\mathbf v_{n,m}^{(h-1)})\approx \mathbf w_n^{(h-1)}$ and $\psi_n^{(h-1)}(\mathbf v_{n,m}^{(h-1)})\approx (\mathbf w_n^{(h-1)})^{-1}$. For each $(h-1)$-tuple $(i_1,i_2,\dots,i_{h-1})$ with $i_1,i_2,\dots,i_{h-1}\in\{1,2,\dots,2n\}$, we can write
\begin{align}
\varphi_n^{(h-1)}(x_{i_1i_2\dots i_{h-1}})&=x^\delta_{\alpha_1\alpha_2\dots\alpha_{h-1}},\label{eq:phi}\\
\psi_n^{(h-1)}(x_{i_1i_2\dots i_{h-1}})&=x^\varepsilon_{\beta_1\beta_2\dots\beta_{h-1}},\label{eq:psi}
\end{align}
where $\delta,\varepsilon\in\{1,-1\}$ and $\alpha_1,\alpha_2,\dots,\alpha_{h-1},\beta_1,\beta_2,\dots,\beta_{h-1}\in\{1,2,\dots,n\}$ are uniquely determined by $(i_1,i_2,\dots,i_{h-1})$. Now we define the substitutions $\varphi_n^{(h)}$ and $\psi_n^{(h)}$ from $X_{2n}^{(h)}$ onto $\overline{X}_n^{(h)}$ as follows: for all $i_1,i_2,\dots,i_{h-1},i_h\in\{1,2,\dots,2n\}$,
\begin{align}
\varphi_n^{(h)}(x_{i_1i_2\dots i_{h-1}i_h})&:=\begin{cases}
x^\delta_{\alpha_1\alpha_2\dots\alpha_{h-1}i_h}&\text{if }1\le i_h\le n,\\
x^\varepsilon_{\beta_1\beta_2\dots\beta_{h-1}(i_h-n)}&\text{if }n+1\le i_h\le 2n,
\end{cases}\label{eq:phinew}\\
\psi_n^{(h)}(x_{i_1i_2\dots i_{h-1}i_h})&:=\begin{cases}
x^\delta_{\alpha_1\alpha_2\dots\alpha_{h-1}(n+1-i_h)}&\text{if }1\le i_h\le n,\\
x^\varepsilon_{\beta_1\beta_2\dots\beta_{h-1}(2n+1-i_h)}&\text{if }n+1\le i_h\le 2n,
\end{cases}\label{eq:psinew}
\end{align}
where $\delta,\varepsilon\in\{1,-1\}$ and $\alpha_1,\alpha_2,\dots,\alpha_{h-1},\beta_1,\beta_2,\dots,\beta_{h-1}\in\{1,2,\dots,n\}$ are determined by \eqref{eq:phi} and \eqref{eq:psi}. Recall that the words $\mathbf v_{n,i}^{(h-1)}$ and the terms $\mathbf w_{n,i}^{(h-1)}$ are obtained by appending $i$ to the indices of all variables occurring in respectively $\mathbf v_n^{(h-1)}$ and $\mathbf w_n^{(h-1)}$. Therefore, \eqref{eq:phinew} and \eqref{eq:psinew} ensure that for each $i\in\{1,2,\dots,n\}$, the identities  $\varphi_n^{(h-1)}(\mathbf v_{n,m}^{(h-1)})\approx \mathbf w_n^{(h-1)}$ and $\psi_n^{(h-1)}(\mathbf v_{n,m}^{(h-1)})\approx (\mathbf w_n^{(h-1)})^{-1}$ imply the identities
\[
\varphi_{n}^{(h)}(\mathbf v_{n,m,i}^{(h-1)}) \approx  \mathbf w_{n,i}^{(h-1)}\ \text{ and} \ \psi_{n}^{(h)}(\mathbf v_{n,m,n+i}^{(h-1)}) \approx  (\mathbf w_{n,i}^{(h-1)})^{-1}.
\]
Using these, we see that in every inverse semigroup,
\begin{align*}
\varphi_n^{(h)}(\mathbf v_{n,m}^{(h)})&{}\approx \Bigl(\prod_{i=1}^n\mathbf w_{n,i}^{(h-1)}\Bigr)\Bigl(\prod_{i=1}^n(\mathbf w_{n,i}^{(h-1)})^{-1}\Bigr)\Bigl(\Bigl(\prod_{i=1}^n\mathbf w_{n,n-i+1}^{(h-1)}\Bigr)\Bigl(\prod_{i=1}^n(\mathbf w_{n,i}^{(h-1)})^{-1}\Bigr)\Bigr)^{2m-1}\\
&{}\approx \Bigl(\prod_{i=1}^n\mathbf w_{n,i}^{(h-1)}\Bigr)\Bigl(\prod_{i=1}^n(\mathbf w_{n,i}^{(h-1)})^{-1}\Bigr)\Bigl(\prod_{i=1}^n\mathbf w_{n,n-i+1}^{(h-1)}\Bigr)\Bigl(\prod_{i=1}^n(\mathbf w_{n,i}^{(h-1)})^{-1}\Bigr)\approx \mathbf w_n^{(h)}.
\end{align*}
Similarly, $\psi_n^{(h)}(\mathbf v_{n,m}^{(h)})\approx (\mathbf w_n^{(h)})^{-1}$ in every inverse semigroup.

It follows that if an inverse semigroup $(S,\cdot,{}^{-1})$ satisfies the identity $\mathbf v_{n,m}^{(h)}\approx (\mathbf v_{n,m}^{(h)})^2$ then it also satisfies the identity $\mathbf w_n^{(h)}\approx (\mathbf w_n^{(h)})^2$. However, it is easy to see (and is mentioned in~\cite[proof of Theorem~5.1]{Kad03}) that the inverse semigroup $(S_n^{(h)},\cdot,{}^{-1})$ does not satisfy the identity $\mathbf w_n^{(h)}\approx (\mathbf w_n^{(h)})^2$. Indeed, if $\zeta\colon X_n^{(h)}\to S_n^{(h)}$ is defined by $\zeta(x_{i_1i_2\dots i_{h}}):=\chi_{i_1i_2\dots i_{h}}^{-1}$, the transformation $\zeta(\mathbf w_n^{(h)})$ maps $2^hn^h$ to 0 while the transformation $\zeta((\mathbf w_n^{(h)})^2)$ is nowhere defined. Therefore, the identity $\mathbf v_{n,m}^{(h)}\approx (\mathbf v_{n,m}^{(h)})^2$ fails in $(S_n^{(h)},\cdot)$.
\end{proof}

\section{Proofs of main results}
\label{sec:proofs}

We need an observation from \cite{Vol21}. Here it is stated in the notation of the present note.

\begin{lemma}[\!{\mdseries\cite[Lemma~2.1]{Vol21}}]
\label{lem:as-ais}
If an inverse semigroup $(S,\cdot,{}^{-1})$ satisfies for some $p$, the identity $x^p\approx x^{p+1}$, then $(S,\le_{\mathrm{nat}})$ is an $\inf$-semilattice and $x+_{\mathrm{nat}} y=(xy^{-1})^px$ for all $x,y\in S$.
\end{lemma}

Lemma~\ref{lem:as-ais} and the features of Ka\softd{}ourek's construction from Sect.~\ref{sec:kadourek} lead to the following.

\begin{theorem}
\label{thm:main}
Let $\mathcal{S}=(S,+,\cdot)$ be an \ais\ whose multiplicative reduct satisfies the identities \eqref{eq:IDh} for all $n\ge2$ and some $m,h\ge 1$. If the \ais\ $(B_2^1,+_{\mathrm{nat}},\cdot)$ satisfies all identities of $\mathcal{S}$, then $\mathcal{S}$ admits no finite identity basis.
\end{theorem}

\begin{proof}
Arguing by contradiction, assume that for some $k$ the \ais\ $\mathcal{S}$ has an identity basis $\Sigma$ such that each identity in $\Sigma$ involves less than $k$ variables. Consider the inverse semigroup $(S_k^{(h)},\cdot,{}^{-1})$ from Sect.~\ref{sec:kadourek} where $h$ is the parameter of the identities \eqref{eq:IDh} that hold in  the multiplicative reduct $(S,\cdot)$ of $\mathcal{S}$. By Corollary~\ref{cor:x2x3} $(S_k^{(h)},\cdot,{}^{-1})$ satisfies $x^2\approx x^3$, and therefore, Lemma~\ref{lem:as-ais} implies that $(S_k^{(h)},+_{\mathrm{nat}},\cdot)$ is an ai-semi\-ring. We claim that this \ais\ satisfies an arbitrary identity $\mathbf u\approx\mathbf v$ in $\Sigma$.

By Lemma~\ref{lem:as-ais} $x+_{\mathrm{nat}} y$ expresses as $(xy^{-1})^2x$ in $(S_k^{(h)},+_{\mathrm{nat}},\cdot)$. Therefore one can rewrite the identity $\mathbf u\approx \mathbf v$ into an identity $\mathbf u'\approx \mathbf v'$ in which $\mathbf u'$ and $\mathbf v'$ are $(\cdot,{}^{-1})$-terms with the same variables as $\mathbf u$ and $\mathbf v$. Let $x_1,x_2,\dots,x_\ell$ be all variables that occur in $\mathbf u'$ or $\mathbf v'$. Consider an arbitrary substitution $\tau\colon\{x_1,x_2,\dots,x_\ell\}\to S_k^{(h)}$ and let $(T,\cdot,{}^{-1})$ be the inverse subsemigroup of $(S_k^{(h)},\cdot,{}^{-1})$ generated by the elements $\tau(x_1),\tau(x_2),\dots,\tau(x_\ell)$. Since $\ell<k$, Proposition~\ref{prop:kadourek1} implies that $(T,\cdot,{}^{-1})$ satisfies all identities of the 6-element Brandt monoid  $(B_2^1,\cdot,{}^{-1})$.

Since by the condition of the theorem, the \ais\ $(B_2^1,+_{\mathrm{nat}},\cdot)$ satisfies all identities of $\mathcal{S}$, the identity $\mathbf u\approx\mathbf v$ holds in $(B_2^1,+_{\mathrm{nat}},\cdot)$. This implies that the rewritten identity $\mathbf u'\approx \mathbf v'$ holds in $(B_2^1,\cdot,{}^{-1})$. (Here we utilize the fact that $(B_2^1,\cdot,{}^{-1})$ satisfies $x^2\approx x^3$, and therefore, $x+_{\mathrm{nat}} y$ expresses in $(B_2^1,\cdot,{}^{-1})$ as the same $(\cdot,{}^{-1})$-term $(xy^{-1})^2x$.) Hence the identity $\mathbf u'\approx \mathbf v'$ holds also in the inverse semigroup $(T,\cdot,{}^{-1})$, and so $\mathbf u'$ and $\mathbf v'$ take the same value under every substitution of elements of $T$ for the variables $x_1,\dots,x_\ell$. In particular, $\tau(\mathbf u)=\tau(\mathbf u')=\tau(\mathbf v')=\tau(\mathbf v)$. Since the substitution $\tau$ is arbitrary, this proves our claim that the identity $\mathbf u\approx \mathbf v$ holds in the \ais\ $(S_k^{(h)},+_{\mathrm{nat}},\cdot)$. Since $\mathbf u\approx \mathbf v$ is an arbitrary identity from the identity basis $\Sigma$ of $\mathcal{S}$, we see that $(S_k^{(h)},+_{\mathrm{nat}},\cdot)$ satisfies all identities of $\mathcal{S}$. Forgetting the addition, we conclude that the multiplicative reduct $(S_k^{(h)},\cdot)$  of $(S_k^{(h)},+_{\mathrm{nat}},\cdot)$ satisfies all identities of the multiplicative reduct $(S,\cdot)$ of $\mathcal{S}$. By the condition of the theorem, $(S,\cdot)$ satisfies the identity $\mathbf v_{k,m}^{(h)} \approx (\mathbf v_{k,m}^{(h)})^2$ for some  $m\ge 1$, but by Proposition \ref{prop:Sn(h)} this identity fails in $(S_k^{(h)},\cdot)$, a contradiction.
\end{proof}

\begin{remark}
\label{rm:zero}
Since the \ais{}\ $(S_k^{(h)},+_{\mathrm{nat}},\cdot)$ and $(B_2^1,+_{\mathrm{nat}},\cdot)$ used in the above proof are semirings with 0, the same proof works fine for \ais{}s with 0 treated as algebras of type (2,2,0). The same conclusion applies to all corollaries of Theorem~\ref{thm:main} stated below.
\end{remark}

\begin{remark}
\label{rm:natural}
The multiplicative reduct of the \ais\ $\mathcal{S}$ in Theorem~\ref{thm:main} need not be an inverse semigroup, and in a follow up paper, we will give some applications of Theorem~\ref{thm:main} to \ais{}s whose multiplicative reducts are \emph{block-groups} in the sense of~\cite{Pin95}. Moreover, even the reduct is inverse, $\mathcal{S}$ need not be a naturally semilattice-ordered inverse semigroup. Observe that an inverse semigroup may admit more than one addition making it an \ais. As a concrete example, borrowed from \cite{Vol21}, consider the ai-semiring $(\Sigma_7,+,\cdot)$ introduced in~\cite{Dolinka07}. Here the set $\Sigma_7$ consists of the following Boolean $2\times 2$-matrices:
\[
\begin{pmatrix} 1&1\\1&1\end{pmatrix},\ \begin{pmatrix} 1&0\\0&1\end{pmatrix},\ \begin{pmatrix} 1&0\\1&1\end{pmatrix},\ \begin{pmatrix} 1&1\\0&1\end{pmatrix},\  \begin{pmatrix} 0&1\\1&1\end{pmatrix},\ \begin{pmatrix} 1&1\\1&0\end{pmatrix},\ \begin{pmatrix} 0&0\\0&0\end{pmatrix},
\]
and the operations + and $\cdot$ are the usual addition and multiplication of Boolean matrices. The multiplicative reduct $(\Sigma_7,\cdot)$ is easily seen to be a combinatorial inverse semigroup. Therefore, one can define the ``natural'' addition $+_{\mathrm{nat}}$ on $\Sigma_7$ via \eqref{eq:inf}, but the addition is quite different from the addition of Boolean matrices. Moreover, the \ais{}s $(\Sigma_7,+,\cdot)$ and $(\Sigma_7,+_{\mathrm{nat}},\cdot)$ even fail to be equationally equivalent as is witnessed, for instance, by the identity $(xy+yx)^2\approx x^2+y^2$ that holds in $(\Sigma_7,+_{\mathrm{nat}},\cdot)$ but not in $(\Sigma_7,+,\cdot)$.

In general, the question of how the equational properties of two \ais{}s may relate when the \ais{}s have the same multiplicative reduct appears to be non-trivial and worth exploration. In the above example, both $(\Sigma_7,+,\cdot)$ and $(\Sigma_7,+_{\mathrm{nat}},\cdot)$ are NFB. We do not know if there exists a finite inverse semigroup $(S,\cdot,{}^{-1})$ that admits two additions $+_1$ and $+_2$ such that both $(S,+_{1},\cdot)$ and $(S,+_{2},\cdot)$ are \ais{}s, but only one of them is NFB.
\end{remark}

It is easy to deduce Theorem~\ref{thm:comb-nfb} from Theorem~\ref{thm:main} but in fact, our proof technique gives a more general result that we state first.

\begin{theorem}
\label{thm:comb-nfb-gen}
Let $\mathcal{S}=(S,+,\cdot)$ be a finite \ais\ whose  multiplicative reduct $(S,\cdot)$ is an inverse semigroup with nilpotent subgroups. If the \ais\ $(B_2^1,+_{\mathrm{nat}},\cdot)$ satisfies all identities of $\mathcal{S}$, then $\mathcal{S}$ admits no finite identity basis.
\end{theorem}

\begin{proof}
If $(S,\cdot)$ contains a non-abelian nilpotent subgroup, then $\mathcal{S}$ admits no finite identity basis by~\cite[Theorem~6.1]{JackRenZhao}. So we may assume that every subgroup of $(S,\cdot)$ is abelian. Since $\mathcal{S}$ is finite, there is some $m\ge 1$ that the exponent of every subgroup of $(S,\cdot)$ divides $m$ and for some $h>1$, there exists a principal series \eqref{eq:series} in $(S,\cdot)$. Thus, $(S,\cdot)$ is an $(h,m)$-semi\-group. By Proposition~\ref{prop:id-fh+1}, $(S,\cdot)$ satisfies the identity $\mathbf v_{n,m}^{(h+1)} \approx (\mathbf v_{n,m}^{(h+1)})^2$ for all $n\ge 2$, and therefore, Theorem~\ref{thm:main} applies.
\end{proof}

We are ready to prove Theorem~\ref{thm:comb-nfb}. Recall its statement: if $(B_2^1,\cdot,{}^{-1})$ satisfies all identities of a finite combinatorial inverse semigroup $(S,\cdot,{}^{-1})$, then the \ais\ $(S,+_{\mathrm{nat}},\cdot)$ admits no finite identity basis.

\begin{proof}[of Theorem~\ref{thm:comb-nfb}]
In view of Theorem~\ref{thm:comb-nfb-gen}, it remains to verify that if the 6-element Brandt monoid $(B_2^1,\cdot,{}^{-1})$ satisfies all identities of a finite combinatorial inverse semigroup $(S,\cdot,{}^{-1})$, then the \ais\ $(B_2^1,+_{\mathrm{nat}},\cdot)$ satisfies every identity of the \ais\ $(S,+_{\mathrm{nat}},\cdot)$.

Since $(S,\cdot,{}^{-1})$ is finite and combinatorial, it satisfies the identity $x^p\approx x^{p+1}$ for some $p$, and we may assume that $p\ge 2$. By Lemma~\ref{lem:as-ais} we have $x+_{\mathrm{nat}} y=(xy^{-1})^px$ for all $x,y\in S$. Since $(B_2^1,\cdot,{}^{-1})$ satisfies $x^2\approx x^3$, we may assume that $x+_{\mathrm{nat}} y$ expresses in $(B_2^1,\cdot,{}^{-1})$ as the same $(\cdot,{}^{-1})$-term $(xy^{-1})^px$. Now take any identity $\mathbf u\approx\mathbf v$ holding in $(S,+_{\mathrm{nat}},\cdot)$ and rewrite it into an identity $\mathbf u'\approx \mathbf v'$ in which $\mathbf u'$ and $\mathbf v'$ are $(\cdot,{}^{-1})$-terms. The latter identity then holds in $(B_2^1,\cdot,{}^{-1})$ and rewriting it back to $\mathbf u\approx\mathbf v$, we see that $\mathbf u\approx\mathbf v$ holds in $(B_2^1,+_{\mathrm{nat}},\cdot)$.
\end{proof}

Finally, we prove Theorem~\ref{thm:rook}. Recall that it states that the \ais\ $(R_t,+_{\mathrm{nat}},\cdot)$ built from the rook monoid $\mathcal{R}_t$ admits a finite identity basis if and only if $t=1$.

\begin{proof}[of Theorem~\ref{thm:rook}]
The rook monoid $\mathcal{R}_1$ is actually the 2-element semilattice $(Y_2,\cdot)$. We have already mentioned that the \ais\ $(Y_2,+_{\mathrm{nat}},\cdot)$ is finitely based; see Remark~\ref{rm:almostall}.

By Proposition~\ref{prop:rook2and3}(1) the rook monoid $\mathcal{R}_2$ satisfies the identity $\mathbf v_{n,2}^{(2)} \approx (\mathbf v_{n,2}^{(2)})^2$ for any $n\ge 2$. The \ais\ $(B_2^1,+_{\mathrm{nat}},\cdot)$ satisfies all identities of $(R_2,+_{\mathrm{nat}},\cdot)$ just because the former semiring is a subsemiring of the latter: to get the set $R_2$ of all zero-one $2\times 2$-matrices with at most one 1 in each row and column, one only has to add the matrix $\begin{pmatrix}0&1\\1&0\end{pmatrix}$ to the six matrices in~\eqref{eq:b21}. Thus, Theorem~\ref{thm:main} applies to $(R_2,+_{\mathrm{nat}},\cdot)$.

By Proposition~\ref{prop:rook2and3}(2) the rook monoid $\mathcal{R}_3$ satisfies the identity $\mathbf v_{n,6}^{(4)} \approx (\mathbf v_{n,6}^{(4)})^2$ for any $n\ge 2$. Clearly, the \ais\ $(R_2,+_{\mathrm{nat}},\cdot)$ embeds into $(R_3,+_{\mathrm{nat}},\cdot)$ whence $(B_2^1,+_{\mathrm{nat}},\cdot)$ satisfies all identities of $(R_3,+_{\mathrm{nat}},\cdot)$. Again, Theorem~\ref{thm:main} applies to $(R_3,+_{\mathrm{nat}},\cdot)$.

Finally, if $t\ge4$, the rook monoid $\mathcal{R}_t$ has the symmetric group $\mathrm{Sym}_t$ as its group of units. The group $\mathrm{Sym}_t$ with $t\ge4$ possesses non-abelian nilpotent subgroups, for instance, the dihedral group of order 8. By~\cite[Theorem~6.1]{JackRenZhao} the \ais\ $(R_t,+_{\mathrm{nat}},\cdot)$ admits no finite identity basis.
\end{proof}


\begin{thebibliography}{99}
\bibitem{AEI03}
Aceto, L., \'Esik, Z., Ing\'olfsd\'ottir, A.: The max-plus algebra of the natural numbers has no finite equational basis, Theoret. Comput. Sci. \textbf{293}, 169--188 (2003)

\bibitem{AM11}
Andr\'eka, H., Mikul\'as, Sz.: Axiomatizability of positive algebras of binary relations, Algebra Universalis \textbf{66}, 7--34 (2011)

\bibitem{Dolinka07}
Dolinka, I.: A nonfinitely based finite semiring. Internat. J. Algebra Comput. \textbf{17}, 1537--1551 (2007)

\bibitem{Dolinka09a}
Dolinka, I.: A class of inherently nonfinitely based semirings. Algebra Universalis \textbf{60}, 19--35 (2009)

\bibitem{Dolinka09b}
Dolinka, I.:  A remark on nonfinitely based semirings, Semigroup Forum \textbf{78}, 368--373 (2009)

\bibitem{Dolinka09c}
Dolinka, I.: The finite basis problem for endomorphism semirings of finite semilattices with zero, Algebra Universalis \textbf{61}, 441--448 (2009)

\bibitem{JackRenZhao}
Jackson, M., Ren, Miaomiao, Zhao, Xianzhong: Nonfinitely based ai-semirings with finitely based semigroup reducts, Preprint, see \url{https://arxiv.org/abs/2112.13918v1} (2021)

\bibitem{JackStokes}
Jackson, M., Stokes, T.: Identities in the algebra of partial maps, Internat. J. Algebra Comput. \textbf{16}, 1131--1159 (2006)

\bibitem{JKM09}
Je\v{z}ek, J., Kepka, T., Mar\'oti, M.: The endomorphism semiring of a semilattice. Semigroup Forum \textbf{78}, 21--26 (2009).

\bibitem{Jip17}
Jipsen, P.: Relation algebras, idempotent semirings and generalized bunched implication algebras. In: P.~H\"ofner, D. Pous, G. Struth (eds.), Relational and Algebraic Methods in Computer Science, Lecture Notes in Computer Science, vol. 10226, pp. 144--158. Springer, Cham (2017)

\bibitem{Kad91}
Ka\softd{}ourek, J.: On varieties of combinatorial inverse semigroups. I. Semigroup Forum \textbf{43}, 305--330 (1991)

\bibitem{Kad03}
Ka\softd{}ourek, J.: On bases of identities of finite inverse semigroups with solvable subgroups. Semigroup Forum \textbf{67}, 317--343 (2003)

\bibitem{Law99}
Lawson, M.V.: Inverse Semigroups. The Theory of Partial Symmetries. World Scientific, Singapore (1999)

\bibitem{Leech95}
Leech, J.: Inverse monoids with a natural semilattice ordering. Proc. London Math. Soc. \textbf{s3-70}(1), 146--182 (1995)

\bibitem{Mash79}
Mashevitzky, G.I.: Identities in Brandt semigroups. In: Semigroup Varieties and Semigroups of Endomorphisms, pp. 126--137. Leningrad State Pedagogical Institute, Leningrad (1979) (Russian)

\bibitem{Munn57}
Munn, W.D.: The characters of the symmetric inverse semigroup. Math. Proc. Cambridge Philos. Soc. 53(1), 13--18 (1957)

\bibitem{Pet84}
Petrich, M.: Inverse Semigroups. John Wiley \& Sons, New York (1984)

\bibitem{Pin95}
Pin, J.-\'E.: $BG=PG$, a success story. In: Fountain, J. (ed.) Semigroups, Formal Languages and Groups. NATO ASI Ser., Ser. C: Math. Phys. Sci., vol. 466, pp. 33--47. Kluwer Academic Publishers, Dordrecht--Boston--London (1995)

\bibitem{Pin98}
Pin, J-\'E.:  Tropical semirings. In: J. Gunawardena (ed.), Idempotency, Publications of the Newton Institute, vol. 11, pp. 50--69. Cambridge University Press, Cambridge (1998).

\bibitem{Polak01}
Pol\'ak, L.: Syntactic semiring of a language. In: J. Sgall, A. Pultr, P. Kolman (eds.), Mathematical Foundations of Computer Science 2001, Lecture Notes in Computer Science, vol. 2136, pp. 611--620.  Springer-Verlag, Berlin-Heidelberg (2001)

\bibitem{Rei08}
Reilly, N.R.: The interval $[\mathbf{B}_2, \mathbf{NB}_2]$ in the lattice of Rees--Sushkevich varieties. Algebra Universalis \textbf{59}(3-4), 345--363 (2008)

\bibitem{Sch73}
Schein, B.M.: Completions, translational hulls and ideal extensions of inverse semigroups. Czechoslovak Math. J. \textbf{23}(4), 575--610 (1973)

\bibitem{Sol02}
Solomon, L.: Representations of the rook monoid. J. Algebra \textbf{256}(2), 309--342 (2002)


\bibitem{Vol21}
Volkov, M.V.: Semiring identities of the Brandt monoid. Algebra Universalis \textbf{82}, Article no. 42 (2021)
\end{thebibliography}
\end{document}